\documentclass[12pt]{amsart}

\usepackage[margin=1in]{geometry}
\usepackage{graphicx}
\usepackage{paralist}
\usepackage{mathdots}

\newcommand{\real}{\mathbb{R}}

\usepackage{url}

\DeclareMathOperator{\diag}{diag}

\newtheorem{theorem}{Theorem}[section]
\newtheorem{corollary}{Corollary}[section]

\newtheorem{lemma}{Lemma}[section]
\theoremstyle{definition}

\newtheorem{example}{Example}[section]
\newtheorem{remark}{Remark}[section]
\newtheorem{conjecture}{Conjecture}[section]

\begin{document}

\title{Spectral characterizations of anti-regular graphs}

\author{Cesar O. Aguilar}
\address{Department of Mathematics, State University of New York, Geneseo}
\email{aguilar@geneseo.edu}

\author{Joon-yeob Lee}
\address{Department of Mathematics, State University of New York, Geneseo}
\email{jl56@geneseo.edu}

\author{Eric Piato}
\address{Department of Mathematics, State University of New York, Geneseo}
\email{esp6@geneseo.edu}

\author{Barbara J. Schweitzer}
\address{Department of Mathematics, State University of New York, Geneseo}
\email{bjs22@geneseo.edu}

\begin{abstract}
We study the eigenvalues of the unique connected anti-regular graph $A_n$.  Using Chebyshev polynomials of the second kind, we obtain a trigonometric equation whose roots are the eigenvalues and perform elementary analysis to obtain an almost complete characterization of the eigenvalues.  In particular, we show that the interval $\Omega=[\tfrac{-1-\sqrt{2}}{2}, \tfrac{-1+\sqrt{2}}{2}]$ contains only the trivial eigenvalues $\lambda = -1$ or $\lambda=0$, and any closed interval strictly larger than $\Omega$ will contain eigenvalues of $A_n$ for all $n$ sufficiently large.  We also obtain bounds for the maximum and minimum eigenvalues, and for all other eigenvalues we obtain interval bounds that improve as $n$ increases.  Moreover, our approach reveals a more complete picture of the bipartite character of the eigenvalues of $A_n$, namely, as $n$ increases the eigenvalues are (approximately) symmetric about the number $-\tfrac{1}{2}$.  We also obtain an asymptotic distribution of the eigenvalues as $n\rightarrow\infty$.  Finally, the relationship between the eigenvalues of $A_n$ and the eigenvalues of a general threshold graph is discussed.
\end{abstract}

\maketitle


\baselineskip 1.5em

\section{Introduction}
Let $G=(V,E)$ be an $n$-vertex simple graph, that is, a graph without loops or multiple edges, and let $\deg_G(v)$ denote the degree of $v\in V$.  It is an elementary exercise to show that $G$ contains at least two vertices of equal degree.  If $G$ has all vertices with equal degree then $G$ is called a \textit{regular} graph.  We say then that $G$ is an \textit{anti-regular} graph if $G$ has only two vertices of equal degree.  If $G$ is anti-regular it follows easily that the complement graph $\overline{G}$ is also anti-regular since $\deg_{G}(v) = (n-1) - \deg_{\overline{G}}(v)$.  It was shown in \cite{MB-GC:67} that up to isomorphism, there is only one connected anti-regular graph on $n$ vertices and that its complement is the unique disconnected $n$-vertex anti-regular graph.  Let us denote by $A_n$ the unique connected anti-regular graph on $n\geq 2$ vertices.  The graph $A_n$ has several interesting properties.  For instance, it was shown in \cite{RM:03} that $A_n$ is \textit{universal for trees}, that is, every tree graph on $n$ vertices is isomorphic to a subgraph of $A_n$.  Anti-regular graphs are \textit{threshold graphs} \cite{NM-UP:95} which have numerous applications in computer science and psychology.  Within the family of threshold graphs, the anti-regular graph is uniquely defined by its \textit{independence polynomial} \cite{VL-EM:12}.  Also, the eigenvalues of the Laplacian matrix of $A_n$ are all distinct integers and the missing eigenvalue from $\{0,1,\ldots, n\}$ is $\lfloor (n+1)/2\rfloor$.  In \cite{EM:09}, the characteristic and matching polynomial of $A_n$ are studied and several recurrence relations are obtained for these polynomials, along with some spectral properties of the adjacency matrix of $A_n$.

In this paper, we study the eigenvalues of the adjacency matrix of $A_n$.  If $V(G) = \{v_1, \ldots, v_n\}$ is the vertex set of the graph $G$ then the adjacency matrix of $G$ is the $n\times n$ symmetric matrix $A$ with entry $A(i,j) = 1$ if $v_i$ and $v_j$ are adjacent and $A(i,j) = 0$ otherwise.  From now on, whenever we refer to the eigenvalues of a graph we mean the eigenvalues of its adjacency matrix.  It is known that the eigenvalues of $A_n$ have algebraic multiplicity equal to one and take on a bipartite character \cite{EM:09} in the sense that if $n$ is even then half of the eigenvalues are negative and the other half are positive, and if $n$ is odd then $\lambda = 0$ is an eigenvalue and half of the remaining eigenvalues are positive and the other half are negative.  Our approach to studying the eigenvalues of $A_n$ relies on a natural labeling of the vertices that results in a block triangular structure for the inverse adjacency matrix.  The blocks are tridiagonal pseudo-Toeplitz matrices and Hankel matrices.  We are then able to employ the connection between tridiagonal Toeplitz matrices and Chebyshev polynomials to obtain a trigonometric equation whose roots are the eigenvalues.  Performing elementary analysis on the roots of the equation we obtain an almost complete characterization of the eigenvalues of $A_n$.  In particular, we show that the only eigenvalues contained in the closed interval $\Omega=[\tfrac{-1-\sqrt{2}}{2}, \tfrac{-1+\sqrt{2}}{2}]$ are the trivial eigenvalues $\lambda = -1$ or $\lambda=0$, and any closed bounded interval strictly larger than $\Omega$ will contain eigenvalues of $A_n$ for all $n$ sufficiently large.  This improves a result in \cite{DJ-VT-FT:15} obtained for general threshold graphs and we conjecture that $\Omega$ is a forbidden eigenvalue interval for all threshold graphs (besides the trivial eigenvalues $\lambda = 0$ or $\lambda =-1$).  We also obtain bounds for the maximum and minimum eigenvalues, and for all other eigenvalues we obtain interval bounds that improve as $n$ increases.  Moreover, our approach reveals a more complete picture of the bipartite character of the eigenvalues of $A_n$, namely, as $n$ increases the non-trivial eigenvalues are (approximately) symmetric about the number $-\tfrac{1}{2}$.  Lastly, we obtain an asymptotic distribution of the eigenvalues as $n\rightarrow\infty$.  We conclude the paper by arguing that a characterization of the eigenvalues of $A_n$ will shed light on the broader problem of characterizing the spectrum of general threshold graphs.

\section{Main results}
It is known that the eigenvalues of $A_n$ are simple and that $\lambda=-1$ is an eigenvalue if $n$ is even and $\lambda=0$ is an eigenvalue if $n$ is odd \cite{EM:09}.  In either case, we will call $\lambda = -1$ or $\lambda = 0$ the \textit{trivial eigenvalue} of $A_n$ and will be denoted by $\lambda_0$.  Throughout this paper, we denote the positive eigenvalues of $A_n$ as
\[
\lambda^+_1 < \lambda^+_2 < \cdots < \lambda^+_k
\]
and the negative eigenvalues (excluding $\lambda_0$) as
\[
\lambda^-_{k-1} < \lambda^-_{k-2} < \cdots < \lambda^-_1
\]
if $n=2k$ is even  and 
\[
\lambda^-_{k} < \lambda^-_{k-1} < \cdots < \lambda^-_1
\]
if $n=2k+1$ is odd.  The eigenvalues are labeled this way because $\{\lambda^+_j, \lambda^-_j\}$ should be thought of as a pair for $j\in \{1,2,\ldots,k-1\}$.  In \cite{DJ-VT-FT:15}, it is proved that a threshold graph has no eigenvalue in the interval $(-1,0)$.  Our first result supplies a forbidden interval for the non-trivial eigenvalues of $A_n$.  

\begin{theorem}\label{thm:interval-bnd}
Let $A_n$ denote the connected anti-regular graph with $n$ vertices.  The only eigenvalue of $A_n$ in the interval $\Omega = [\tfrac{-1-\sqrt{2}}{2}, \tfrac{-1+\sqrt{2}}{2}]$ is $\lambda_0\in\{-1,0\}$.
\end{theorem}

Based on numerical experimentation, and our observations in Section~\ref{sec:general-threshold}, we make the following conjectures.
\begin{conjecture}
For any $n$, the anti-regular graph $A_n$ has the smallest positive eigenvalue and has the largest non-trivial negative eigenvalue among all threshold graphs on $n$ vertices.  
\end{conjecture}
By Theorem~\ref{thm:interval-bnd}, a proof of the previous conjecture would also prove the following.
\begin{conjecture}
Other than the trivial eigenvalues $\{0,-1\}$, the interval $\Omega=[\tfrac{-1-\sqrt{2}}{2}, \tfrac{-1+\sqrt{2}}{2}]$ does not contain an eigenvalue of any threshold graph.\footnote{During the publication process of this paper, we were notified that both conjectures have been proved by E. Ghorbani; see  \url{https://arxiv.org/abs/1807.10302}.}
\end{conjecture}

Our next result establishes the asymptotic behavior of the eigenvalues of smallest magnitude as $n\rightarrow\infty$.
\begin{theorem}\label{thm:min-eig}
Let $A_n$ be the connected anti-regular graph with $n=2k$ if $n$ is even and $n=2k+1$ if $n$ is odd.  Let $\lambda^+_1(k)$ denote the smallest positive eigenvalue of $A_n$ and let $\lambda^-_1(k)$ denote the negative eigenvalue of $A_n$ closest to the trivial eigenvalue $\lambda_0$.  The following hold:
\begin{enumerate}[(i)]
\item The sequence $\{\lambda^+_1(k)\}_{k=1}^\infty$ is strictly decreasing and converges to $\tfrac{-1+\sqrt{2}}{2}$.  
\item The sequence $\{\lambda^-_1(k)\}_{k=1}^\infty$ is strictly increasing and converges to $\tfrac{-1-\sqrt{2}}{2}$.  
\end{enumerate}
\end{theorem}
As a result, the interval $\Omega = [\tfrac{-1-\sqrt{2}}{2}, \tfrac{-1+\sqrt{2}}{2}]$ in Theorem~\ref{thm:interval-bnd} is best possible in the sense that any closed bounded interval strictly larger than $\Omega$ will contain eigenvalues of $A_n$ (other than the trivial eigenvalue) for all sufficiently large $n$.

Our next main result says that $\lambda^+_j + \lambda^-_j + 1\approx 0$ for almost all $j\in \{1,2,\ldots,k-1\}$ provided that $k$ is sufficiently large.  In other words, the eigenvalues are approximately symmetric about the number $-\tfrac{1}{2}$. 
\begin{theorem}\label{thm:large-k}
Let $A_n$ be the connected anti-regular graph where $n=2k$ or $n=2k+1$.  Fix $r\in (0,1)$ and let $\varepsilon > 0$ be arbitrary.  Then for $k$ sufficiently large,
\[
| \lambda^+_j + \lambda^-_j + 1 | < \varepsilon
\]
for all $j\in\{1,2,\ldots,k-1\}$ such that $\frac{2 j}{2k-1} \leq r$ if $n$ is even and $\frac{j}{k}\leq r$ if $n$ is odd.
\end{theorem}
Note that the proportion of integers $j\in\{1,2,\ldots,k-1\}$ that satisfy the inequality in Theorem~\ref{thm:large-k} is $r$.  Hence, Theorem~\ref{thm:large-k} implies that as $k$ increases a larger proportion of the eigenvalues are (approximately) symmetric about the point $-\tfrac{1}{2}$.  Lastly, we obtain an asymptotic distribution of the eigenvalues of all anti-regular graphs.
\begin{theorem}\label{thm:asymptotic-eig}
Let $\sigma(n)$ denote the set of the eigenvalues of $A_n$, let $\sigma = \bigcup_{n\geq 1} \sigma(n)$, and let $\bar{\sigma}$ denote the closure of $\sigma$.
Then
\[
\bar{\sigma} = (-\infty, \tfrac{-1-\sqrt{2}}{2}] \cup \{0,-1\}\cup [\tfrac{-1+\sqrt{2}}{2}, \infty).
\]
\end{theorem}
It turns out that if we restrict $n$ to even then $\bar{\sigma} = (-\infty, \tfrac{-1-\sqrt{2}}{2}] \cup \{-1\}\cup [\tfrac{-1+\sqrt{2}}{2}, \infty)$, and if we restrict $n$ to odd then $\bar{\sigma} = (-\infty, \tfrac{-1-\sqrt{2}}{2}] \cup \{0\}\cup [\tfrac{-1+\sqrt{2}}{2}, \infty)$.

\section{Eigenvalues of tridiagonal Toeplitz matrices}
Our study of the eigenvalues of $A_n$ relies on the relationship between the eigenvalues of tridiagonal Toeplitz matrices and Chebyshev polynomials \cite{JM-DH:03, DK-DS-ST:99}, and so we briefly review the necessary background.  The \textit{Chebyshev polynomial of the second kind} of degree $m$, denoted by $U_m(x)$, is the unique polynomial such that
\begin{equation}\label{eqn:Um}
U_m(\cos\theta) = \frac{\sin((m+1)\theta)}{\sin(\theta)}.
\end{equation}
The first several $U_m$'s are $U_0(x) = 1$, $U_1(x) = 2x$, $U_2(x) = 4x^2-1$, and $U_3(x) = 8x^3-4x$.  The sequence of polynomials $\{U_m\}_{m=0}^\infty$ satisfies the three-term recurrence relation
\begin{equation}\label{eqn:Un}
U_m(x) = 2xU_{m-1}(x) - U_{m-2}(x)
\end{equation}
for $m\geq 2$.  From \eqref{eqn:Um}, the zeros $x_1,x_2,\ldots, x_m$ of $U_m(x)$ are easily determined to be
\[
x_j = \cos\left(\frac{j\pi}{m+1} \right),\quad j=1,2,\ldots,m.
\]
Chebyshev polynomials are used extensively in numerical analysis and differential equations and the reader is referred to \cite{JM-DH:03} for a thorough introduction to these interesting polynomials.

A real tridiagonal Toeplitz matrix is a matrix of the form
\[
T = 
\begin{pmatrix}
a&c&&\\
b&\ddots &\ddots&\\
&\ddots&\ddots&c\\
&&b&a
\end{pmatrix}
\]
for $a,b,c\in\real$.  For our purposes, and to simplify the presentation, we assume that $c=b$.  We can then write $T=aI + bM$ where $I$ is the identity matrix and 
\[
M = 
\begin{pmatrix}
0&1&&\\
1&\ddots &\ddots&\\
&\ddots&\ddots&1\\
&&1&0
\end{pmatrix}.
\]
If $\lambda$ is an eigenvalue of $M$ then clearly $a + b\lambda$ is an eigenvalue of $T$.  Let $\phi_m(t) = \det(tI - M)$ denote the characteristic polynomial of the $m\times m$ matrix $M$.    The Laplace expansion of $\phi_m(t)$ along the last row produces the recurrence relation
\[
\phi_m(t) = t \phi_{m-1}(t) - \phi_{m-2}(t)
\]
for $m\geq 2$, with $\phi_0(t) = 1$ and $\phi_1(t) = t$.  It then follows that $\phi_m(t) = U_m(t/2)$.  Indeed, we have that $U_0(t/2) = 1$ and $U_1(t/2) = 2(t/2) = t$, and from the recurrence \eqref{eqn:Un} we have
\[
U_m(t/2) = 2(t/2) U_{m-1}(t/2) - U_{m-2}(t/2) = t U_{m-1}(t/2) - U_{m-2}(t/2).
\]    

\section{The anti-regular graph $A_n$}\label{sec:even-case} 
As already mentioned, the anti-regular graph $A_n$ is an example of a threshold graph.  Threshold graphs were first studied independently by Chv\'{a}tal and Hammer \cite{VC-PH:77} and by Henderson and Zalcstein \cite{PH-YZ:77}.  There exists an extensive literature on the applications and algorithmic aspects of threshold graphs and the reader is referred to \cite{NM-UP:95, MG:04} for a thorough introduction.  A threshold graph $G$ on $n\geq 2$ vertices can be obtained via an iterative procedure as follows.  One begins with a single vertex $v_1$ and at step $i\geq 2$ a new vertex $v_i$ is added that is either connected to all existing vertices (a dominating vertex) or not connected to any of the existing vertices (an isolated vertex).  The iterative construction of $G$ is best encoded with a \textit{binary creation sequence} $b=(b_1,b_2,\ldots,b_n)$ where $b_1=0$ and, for $i\in\{2,\ldots,n\}$, $b_i=1$ if $v_i$ was added as a dominating vertex or $b_i=0$ if $v_i$ was added as an isolated vertex.  The resulting vertex set $V(G)=\{v_1,v_2,\ldots,v_n\}$ that is consistent with the iterative construction of $G$ will be called the \textit{canonical labeling} of $G$.  In the canonical labeling, the adjacency matrix of $G$ takes the form
 \begin{equation}\label{eqn:adj}
    A = 
    \begin{pmatrix}
    0 &b_2 &b_3&\cdots &b_{n-1} &b_n\\
    b_2 &0 &b_3&\cdots & \vdots &\vdots \\
    b_3&b_3&0&\cdots & \vdots & \vdots\\
    \vdots & \vdots & \vdots& \ddots & b_{n-1} &\vdots\\
    b_{n-1} &\cdots &\cdots &b_{n-1} &0 &b_n\\
    b_n &\cdots &\cdots &\cdots & b_n &0
    \end{pmatrix}.
\end{equation}  
For the anti-regular graph $A_n$, the associated binary sequence is $b=(0,1,0,1,\ldots,0,1)$ if $n$ is even and is $b=(0,0,1,0,1,\ldots,0,1)$ if $n$ is odd.  In what follows, we focus on the case that $n$ is even.  In Section~\ref{sec:odd-case}, we describe the details for the case that $n$ is odd.  

\begin{example}
When $n = 8$ the graph $A_n$ in the canonical labeling is shown in Figure~\ref{fig:A8} and the associated adjacency matrix is 
\begin{equation*}
A = 
\begin{pmatrix}
0 &1 &0 &1 &0 &1 &0 &1\\
1 &0 &0 &1 &0 &1 &0 &1\\
0 &0 &0 &1 &0 &1 &0 &1\\
1 &1 &1 &0 &0 &1 &0 &1\\
0 &0 &0 &0 &0 &1 &0 &1\\
1 &1 &1 &1 &1 &0 &0 &1\\
0 &0 &0 &0 &0 &0 &0 &1\\
1 &1 &1 &1 &1 &1 &1 &0
\end{pmatrix}.
\end{equation*}

\begin{figure}
\centering
\includegraphics[width=100mm,keepaspectratio]{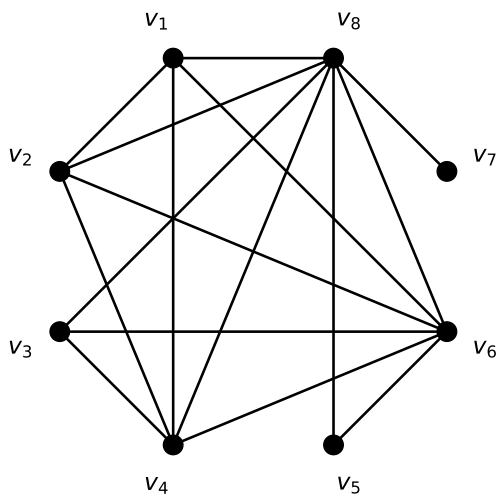}
\caption{The connected anti-regular graph $A_8$ in the canonical labeling}\label{fig:A8}
\end{figure}

\end{example}

As will be seen, a distinct labeling of the vertex set of $A_n$ results in a block structure for $A$.  Let $J=J_k$ denote the $k\times k$ all ones matrix and let $I = I_k$ denote the $k\times k$ identity matrix.
\begin{lemma}
The adjacency matrix of $A_{2k}$ can be written as
\begin{equation}\label{eqn:A}
A =
    \begin{pmatrix}
    0 & B \\[2ex]
    B & J-I
    \end{pmatrix}
\end{equation}
where $B$ is the $k\times k$ Hankel matrix
\[
B = 
\begin{pmatrix}
&&&&1\\
&&&\iddots &1\\
&&\iddots&\iddots & \vdots\\
&\iddots&\iddots&&\vdots\\
1&1&\cdots&\cdots&1
\end{pmatrix}.
\]
\end{lemma}

\begin{proof}
Recall that $A_n$ is the unique connected graph on $n$ vertices that has exactly only two vertices of the same degree.  Moreover, it is known \cite{MB-GC:67} that the repeated degree of $A_{2k}$ is $k$, that is, the degree sequence of $A_{2k}$ in non-increasing order is
\begin{equation}\label{eqn:deg-seq}
    d(A_n) = \left(n-1, n-2, \ldots, \frac{n}{2}, \frac{n}{2}, \frac{n}{2}-1, \ldots, 2, 1\right).
\end{equation}
It is clear that the degree sequence of the graph with adjacency matrix \eqref{eqn:A} is also \eqref{eqn:deg-seq}.  Since $A_n$ is uniquely determined by its degree sequence the claim holds.  
\end{proof}

\begin{remark}
Starting with the canonically labelled vertex set of $A_n$, the permutation 
\begin{equation}\label{eqn:perm}
\sigma = 
    \begin{pmatrix}
    v_1 &v_2 &v_3 &\ldots &v_{n-2} &v_{n-1} &v_n \\
    v_{\frac{n}{2}} &v_{\frac{n}{2}+1} &v_{\frac{n}{2}-1} &\ldots &v_{n-1} &v_1 &v_n
    \end{pmatrix}
\end{equation}
relabels the vertices of $A_n$ so that its adjacency matrix is transformed from \eqref{eqn:adj} to \eqref{eqn:A} via the permutation matrix associated to $\sigma$.  The newly labelled graph is such that $\textup{deg}(v_i) \leq \textup{deg}(v_{i+1})$.  For example, when $n = 8$ the adjacency matrix \eqref{eqn:A} is
\begin{equation*}
A =
\begin{pmatrix}
0 &0 &0 &0 &0 &0 &0 &1\\
0 &0 &0 &0 &0 &0 &1 &1\\
0 &0 &0 &0 &0 &1 &1 &1\\
0 &0 &0 &0 &1 &1 &1 &1\\
0 &0 &0 &1 &0 &1 &1 &1\\
0 &0 &1 &1 &1 &0 &1 &1\\
0 &1 &1 &1 &1 &1 &0 &1\\
1 &1 &1 &1 &1 &1 &1 &0
\end{pmatrix}.
\end{equation*}
\end{remark}

To study the eigenvalues of $A$ we will obtain an eigenvalue equation for $A^{-1}$.  Expressions for $A^{-1}$ involving sums of certain matrices are known when the vertex set of $A_n$ is canonically labelled \cite{RBT:13}.  On the other hand, our choice of vertex labels for $A_n$ produces a closed-form expression for $A^{-1}$.  The proof of the following is left as a straightforward computation.
\begin{lemma}
Consider the adjacency matrix \eqref{eqn:A} of $A_n$ where $n=2k$.  Then
\[
A^{-1} = \begin{pmatrix} V & W \\[2ex] W & 0 \end{pmatrix}
\]
where $W=B^{-1}$ and $V=-B^{-1}(J-I)B^{-1}$.  Explicitly,
\[
W = 
\begin{pmatrix}
&&&-1&1\\
&&\iddots&\iddots&\\
&\iddots&\iddots&&\\
-1&\iddots&&&\\
1&&&&
\end{pmatrix}\quad\textup{ and }\quad
V= 
\begin{pmatrix}
2&-1&&&\\
-1&\ddots&\ddots&&\\
&\ddots&\ddots&\ddots&\\
&&\ddots&2&-1\\
&&&-1&0
\end{pmatrix}.
\]
\end{lemma}
Notice that $W$ is a Hankel matrix and the $(k-1)\times (k-1)$ leading principal submatrix of $V$ is a tridiagonal Toeplitz matrix.
\begin{example}
For our running example when $n=8$ we have
\[
A^{-1} =  \begin{pmatrix} V & W \\[2ex] W & 0 \end{pmatrix} = 
\begin{pmatrix}
2 &-1 &0 &0 &0 &0 &-1 &1\\
-1 &2 &-1 &0 &0 &-1 &1 &0\\
0 &-1 &2 &-1 &-1 &1 &0 &0\\
0 &0 &-1 &0 &1 &0 &0 &0\\
0 &0 &-1 &1 &0 &0 &0 &0\\
0 &-1 &1 &0 &0 &0 &0 &0\\
-1 &1 &0 &0 &0 &0 &0 &0\\
1 &0 &0 &0 &0 &0 &0 &0
\end{pmatrix}.
\]
\end{example}

\section{The eigenvalues of $A_n$}
Suppose that $z=(x,y)\in\real^{2k}$ is an eigenvector of $A^{-1}$ with eigenvalue $\alpha\in\real$, where $x, y \in\real^k$.  From $A^{-1}z = \alpha z$ we obtain the two equations 
\begin{align*}
Vx + Wy &= \alpha x \\
Wx &= \alpha y
\end{align*}
and after substituting $y=\tfrac{1}{\alpha}Wx$ into the first equation and re-arranging we obtain
\[
(\alpha^2 I - \alpha V - W^2)x = 0.
\]
Clearly, we must have $x\neq 0$.  Let $R(\alpha) = \alpha^2 I -  \alpha V - W^2$ so that $\det(R(t)) = \det(tI - A^{-1})$ is the characteristic polynomial of $A^{-1}$.  It is straightforward to verify that
\[
R(\alpha)
=
\begin{pmatrix}
f(\alpha)&\alpha+1&&&\\
\alpha+1&\ddots&\ddots&&\\
&\ddots&\ddots&\ddots&\\
&&\ddots&f(\alpha)&\alpha+1\\
&&&\alpha+1&\alpha^2-1
\end{pmatrix}
\]
where $f(\alpha)=\alpha^2-2\alpha-2$.  Since it is already known that $\alpha=-1$ is an eigenvalue of $A^{-1}$ (this can easily be seen from the last column or row of $R(\alpha)$),  we consider instead the matrix 
\[
S(\alpha) = \frac{1}{(\alpha+1)}R(\alpha)
=
\begin{pmatrix}
h(\alpha)&1&&&\\
1&\ddots&\ddots&&\\
&\ddots&\ddots&\ddots&\\
&&\ddots&h(\alpha)&1\\
&&&1&\alpha-1
\end{pmatrix}
\]
where $h(\alpha) = \frac{\alpha^2-2\alpha-2}{\alpha+1}$.  Hence, $\alpha\neq -1$ is an eigenvalue of $A^{-1}$ if and only if $\det(S(\alpha)) = 0$.  We now obtain a recurrence relation for $\det(S(\alpha))$.  To that end, notice that the $(k-1)\times (k-1)$ leading principal submatrix of $S(\alpha)$ is a tridiagonal Toeplitz matrix.  Hence, for $m\geq 1$ define
\[
\phi_m(\alpha)
=
\det 
\begin{pmatrix}
h(\alpha)&1&&\\
1&\ddots &\ddots&\\
&\ddots&\ddots&1\\
&&1&h(\alpha)
\end{pmatrix}_{m \times m}.
\]
A straightforward Laplace expansion of $\det( S(\alpha))$ along the last row yields
\[
\det (S(\alpha)) = (\alpha-1)\phi_{k-1}(\alpha)-\phi_{k-2}(\alpha).
\]
Hence, $\alpha\neq -1$ is an eigenvalue of $A^{-1}$ if and only if
\[
(\alpha-1)\phi_{k-1}(\alpha)-\phi_{k-2}(\alpha) = 0.
\]
On the other hand, for $m\geq 2$ the Laplace expansion of $\phi_m(\alpha)$ along the last row produces the recurrence relation
\begin{equation*}
\phi_m(\alpha)=h(\alpha)\phi_{m-1}(\alpha)-\phi_{m-2}(\alpha)
\end{equation*} 
with $\phi_0(\alpha)=1$ and $\phi_1(\alpha)=h(\alpha)$.  We can therefore conclude that $\phi_m(\alpha) = U_m\Big(\frac{h(\alpha)}{2}\Big)$ and thus $\alpha\neq -1$ is an eigenvalue of $A^{-1}$ if and only if
\begin{equation}\label{eqn:Un-alpha}
(\alpha-1)U_{k-1}\left(\frac{h(\alpha)}{2}\right) - U_{k-2}\left(\frac{h(\alpha)}{2}\right) = 0.
\end{equation}
Substituting $\alpha=\frac{1}{\lambda}$ into \eqref{eqn:Un-alpha} and re-arranging yields 
\[
\lambda = \frac{ U_{k-1}(\beta(\lambda)) }{U_{k-1}(\beta(\lambda)) + U_{k-2}(\beta(\lambda))}
\]
where $\beta(\lambda) = \frac{h(1/\lambda)}{2}=\frac{1-2\lambda-2\lambda^2}{2\lambda(\lambda+1)}$.  Recalling the definition \eqref{eqn:Um} of $U_m(x)$, we have proved the following.
\begin{theorem}\label{thm:main-eqn-even}
Let $n=2k$ and let $A_n$ denote the connected anti-regular graph with $n$ vertices.  Then $\lambda$ is an eigenvalue of $A_n$ if and only if
\begin{equation}\label{eqn:alpha}
\lambda = \frac{ \sin(k\theta) }{  \sin(k\theta) + \sin((k-1)\theta) }
\end{equation}
where $\theta  = \arccos\left(\frac{1-2\lambda-2\lambda^2}{2\lambda(\lambda+1)}\right)$.
\end{theorem}

\begin{remark}
In \cite[Theorem 3]{EM:09}, recurrence relations for the characteristic polynomial of the adjacency matrix of $A_n$ involving Chebyshev polynomials are obtained using combinatorial methods.
\end{remark}

We now analyze the character of the solution set of \eqref{eqn:alpha}.  To that end, first define the function 
\[
\theta(\lambda) = \arccos\left(\frac{1-2\lambda-2\lambda^2}{2\lambda(\lambda+1)}\right).
\]
Using the fact that the domain and range of $\arccos$ is $[-1,1]$ and $[0,\pi]$, respectively, it is straightforward to show that the domain and range of $\theta(\lambda)$ is $(-\infty,\tfrac{-1-\sqrt{2}}{2}] \cup [\tfrac{-1+\sqrt{2}}{2},\infty)$ and $[0,\pi)$, respectively.  The graph of $\theta(\lambda)$ is displayed in Figure~\ref{fig:theta-graph}.  
\begin{figure}
\centering
\includegraphics[width=120mm,keepaspectratio]{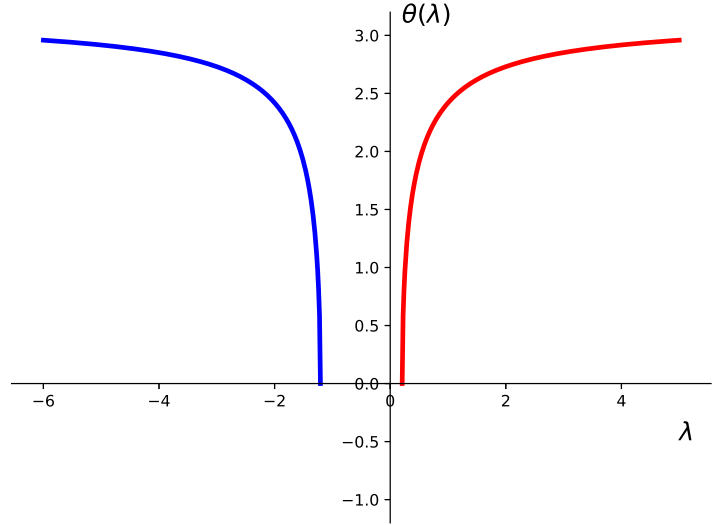}
\caption{Graph of the function $\theta(\lambda) = \arccos\left(\frac{1-2\lambda^2-2\lambda}{2\lambda(\lambda+1)}\right)$ on its domain $(-\infty,\tfrac{-1-\sqrt{2}}{2}] \cup [\tfrac{-1+\sqrt{2}}{2},\infty)$}\label{fig:theta-graph}
\end{figure}
Next, define the function 
\begin{equation}\label{eqn:F}
F(\theta) =  \frac{\sin(k\theta)}{\sin(k\theta)+\sin((k-1)\theta)}.
\end{equation}  
In the interval $(0,\pi)$, the function $F$ has vertical asymptotes at 
\[
\gamma_j = \frac{2j \pi}{2k-1}, \quad j=1,2,\ldots,k-1.
\]
This follows from the trigonometric identity
\[
\sin(k\theta) + \sin((k-1)\theta) = 2\sin{\frac{(2k-1)\theta}{2}}\cos{\frac{\theta}{2}}.
\]
For notational consistency we define $\gamma_0=0$.  Hence, $F$ is continuously differentiable on the set $(0,\gamma_1)\cup (\gamma_1,\gamma_2)\cup \cdots (\gamma_k, \pi)$.   Moreover, using l'H\'{o}pital's rule it is straightforward to show that
\[
\lim_{\theta\rightarrow 0} F(\theta) = \frac{k}{2k-1}
\]
and 
\[
\lim_{\theta\rightarrow \pi } F(\theta) = k.
\]
Hence, there is no harm in defining $F(0)=\frac{k}{2k-1}$ and $F(\pi)=k$ so that we can take $D=[0,\gamma_1)\cup (\gamma_1,\gamma_2)\cup \cdots \cup (\gamma_{k-1}, \pi]$ as the domain of continuity of $F$.

We can now prove Theorem~\ref{thm:interval-bnd}.
\begin{proof}[Proof of Theorem~\ref{thm:interval-bnd}]
The domain of $\theta(\lambda)$ does not contain any point in the interior of $\Omega$ and therefore no solution of \eqref{eqn:alpha} is in the interior of $\Omega$.  At the boundary points of $\Omega$ we have
\[
\theta(\tfrac{-1-\sqrt{2}}{2}) = \theta(\tfrac{-1+\sqrt{2}}{2}) = 0.
\]
On the other hand, $F(0)=\tfrac{k}{2k-1}$ and thus the boundary points of $\Omega$ are not solutions to \eqref{eqn:alpha} either.  The case that $n$ is odd is similar and will be dealt with in Section~\ref{sec:odd-case}.
\end{proof}

We now analyze solutions to \eqref{eqn:alpha} by treating $\theta$ as the unknown variable and expressing $\lambda$ in terms of $\theta$.  To that end, solving for $\lambda$ from the equation $\theta = \arccos\left(\frac{1-2\lambda^2-2\lambda}{2\lambda(\lambda+1)}\right)$ yields the two solutions
\begin{equation}\label{eqn:alphas}
\begin{aligned}
\lambda = f_1(\theta) &= \frac{-(\cos\theta+1)+ \sqrt{(\cos\theta+1)(\cos\theta+3)}}{2(\cos\theta+1)}\\[2ex]
\lambda = f_2(\theta) &= \frac{-(\cos\theta+1) -  \sqrt{(\cos\theta+1)(\cos\theta+3)}}{2(\cos\theta+1)}.
\end{aligned}
\end{equation}
Notice that
\begin{equation}\label{eqn:theta-1}
f_1(\theta) + f_2(\theta) = -1,
\end{equation}
a fact that will be used to show the bipartite character of large anti-regular graphs.  Both $f_1$ and $f_2$ are continuous on $[0,\pi)$, continuously differentiable on $(0,\pi)$, and $\lim_{\theta\rightarrow \pi^-} f_1(\theta) = \infty$ and  $\lim_{\theta\rightarrow \pi^-} f_2(\theta) = -\infty$.  In Figures~\ref{fig:solutions-graph-kmin}-\ref{fig:solutions-graph-kmax}, we plot the functions $f_1(\theta), f_2(\theta)$, and $F(\theta)$ for the values $k=8$ and $k=16$ in the interval $0\leq\theta\leq\pi$.  A dashed line at the value $\lambda = -\tfrac{1}{2} = \frac{f_1(\theta)+f_2(\theta)}{2}$ is included to emphasize that it is a line of symmetry between the graphs of $f_1$ and $f_2$. 

Figures~\ref{fig:solutions-graph-kmin}-\ref{fig:solutions-graph-kmax} show that the graphs of $F$ and $f_1$ intersect exactly $k$ times, say at $\theta^+_1,\ldots,\theta^+_k$, and thus $\lambda^+_j = f_1(\theta^+_j)$   for $j=1,2,\ldots,k$ are the positive eigenvalues of $A_n$.  Similarly, $F$ and $f_2$ intersect exactly $(k-1)$ times, say at $\theta^-_1,\ldots,\theta^-_{k-1}$, and thus $\lambda^-_j=f_2(\theta^-_j)$ for $j=1,2,\ldots,k-1$ are the negative eigenvalues of $A_n$ besides the eigenvalue $\lambda = -1$.  The following theorem formalizes the above observations and supplies interval estimates for the eigenvalues.

\begin{figure}
\centering
\includegraphics[width=120mm,keepaspectratio]{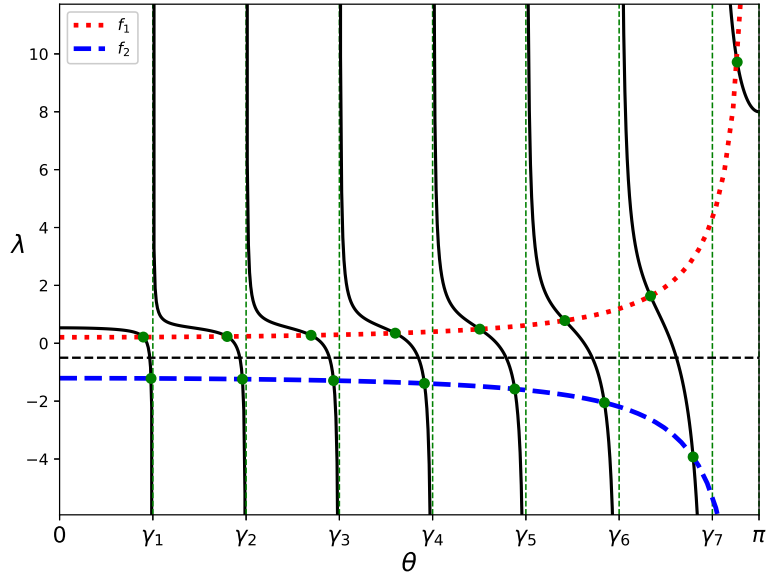}
\caption{Graph of the functions $f_1(\theta), f_2(\theta)$, and $F(\theta)$ (black) for $\theta\in [0,\pi]$ for $k=8$}\label{fig:solutions-graph-kmin}
\end{figure}
\begin{figure}
\centering
\includegraphics[width=120mm,keepaspectratio]{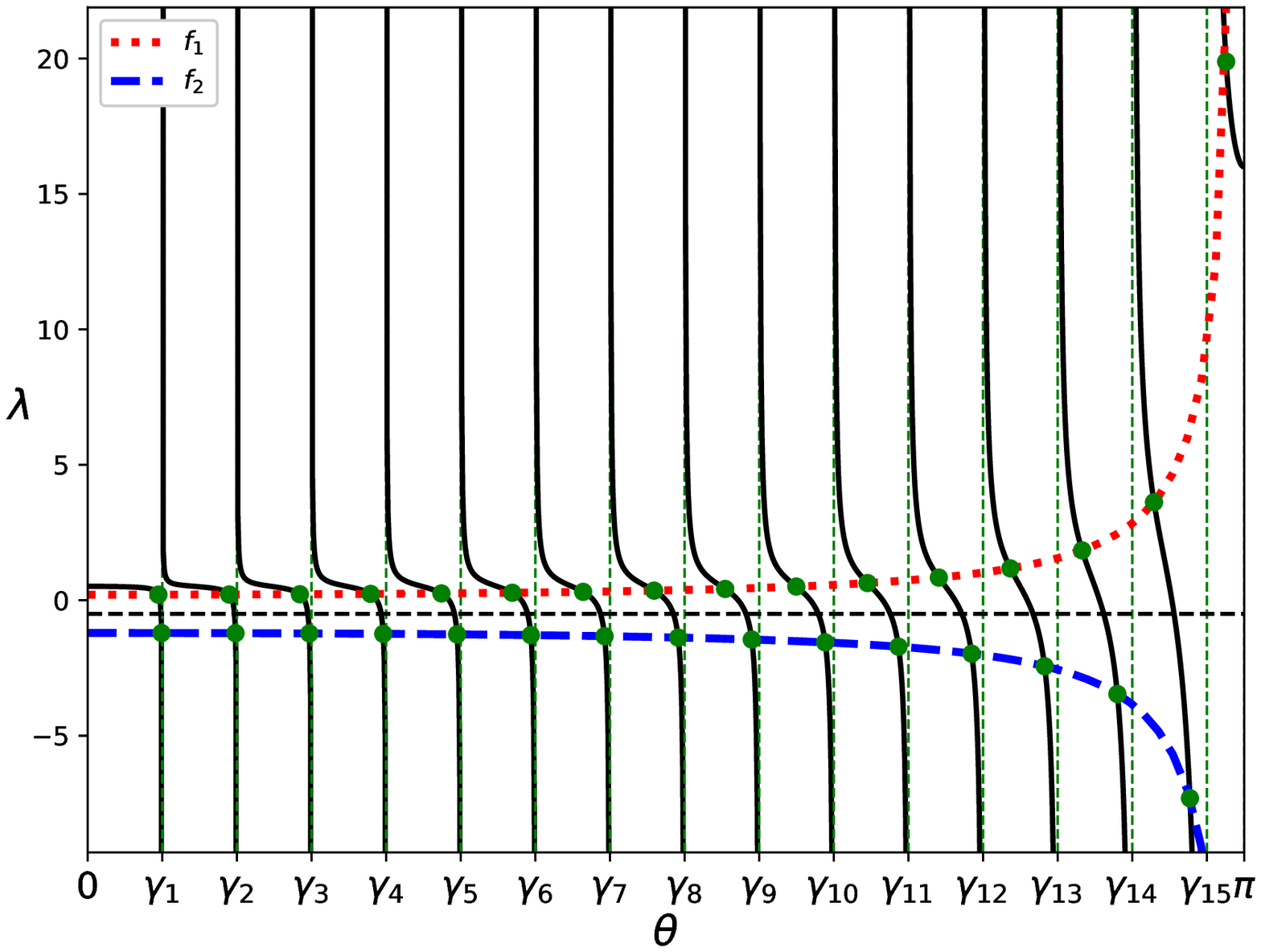}
\caption{Graph of the functions $f_1(\theta), f_2(\theta)$, and $F(\theta)$ (black) for $\theta\in [0,\pi]$ for $k=16$}\label{fig:solutions-graph-kmax}
\end{figure}

\begin{theorem}\label{thm:intersections}
Let $A_n$ be the connected anti-regular graph with $n=2k$ vertices.  Let $F(\theta)$ be defined as in \eqref{eqn:F} and let $f_1(\theta)$ and $f_2(\theta)$ be defined as in \eqref{eqn:alphas}, and recall that $\gamma_j = \frac{2\pi j}{2k-1}$ for $j=0,1,\ldots,k-1$.
\begin{enumerate}[(i)]
\item The functions $F(\theta)$ and $f_1(\theta)$ intersect exactly $k$ times in the interval $0<\theta < \pi$.  If $\theta^+_1 < \theta^+_2 < \cdots < \theta^+_k$ are the intersection points then the positive eigenvalues of $A_n$ are
\[
f_1(\theta^+_1) < f_1(\theta^+_2) < \cdots < f_1(\theta^+_k).
\]
Moreover, for $j=1,2,\ldots, k-1$ it holds that
\[
f_1(\gamma_{j-1}) < f_1(\theta^+_j) < f_1(\gamma_j).
\]
\item The functions $F(\theta)$ and $f_2(\theta)$ intersect exactly $(k-1)$ times in the interval $0 < \theta < \pi$.  If $\theta^-_1 < \theta^-_2 < \cdots < \theta^-_{k-1}$ are the intersection points then the negative eigenvalues of $A_n$ are
\[
f_2(\theta^-_{k-1}) < \cdots < f_2(\theta^-_2) < f_2(\theta^-_1) < -1.
\]
Moreover, for $j=1,2,\ldots,k-1$ it holds that
\[
f_2(\gamma_{j}) < f_2(\theta^-_j) < f_2(\gamma_{j-1}).
\]
\end{enumerate}
\end{theorem}
\begin{proof}
One computes that
\[
f_1'(\theta) = \frac{\sin\theta}{2(\cos\theta+1)\sqrt{(\cos\theta+1)(\cos\theta+3)}}
\]
and thus $f_1'(\theta)>0$ for $\theta\in (0,\pi)$.  Therefore, $f_1$ is strictly increasing on the interval $(0,\pi)$.  Since $f_2(\theta)=-1-f_1(\theta)$ it follows that $f_2$ is strictly decreasing on the interval $(0,\pi)$.  On the other hand, using basic trigonometric identities and the relation $\sin(\theta) U_{k-1}(\cos\theta) = \sin(k\theta)$, we compute that
\begin{align*}
F'(\theta) &= \frac{\left[-k + U_{k-1}(\cos\theta)\cos((k-1)\theta)\right]\sin\theta}{[\sin(k\theta)+\sin((k-1)\theta)]^2}.
\end{align*}
It is known that $\max_{x\in [-1,1]} |U_m(x)| =(m+1)$ and the maximum occurs at $x=\pm 1$ \cite{JM-DH:03}.  Therefore, $F'(\theta) < 0$ for all $\theta\in D\backslash\{0,\pi\}$.  It follows that $F$ is a strictly decreasing function on $D$, and when restricted to the interval $(\gamma_j,\gamma_{j+1})$ for any $j=1,\ldots,k-2$, $F$ is a bijection onto $(-\infty,\infty)$.  Now, since $f_1$ is a strictly increasing continuous function on $[\gamma_j,\gamma_{j+1}]$ for $j=1,2,\ldots,k-2$, the graphs of $F$  and $f_1$ intersect at exactly one point inside the interval $(\gamma_j, \gamma_{j+1})$.  A similar argument applies to $f_2$ and $F$ on each interval $(\gamma_j,\gamma_{j+1})$ for $j=1,2,\ldots,k-2$.  Now consider the leftmost interval $[0,\gamma_1)$.  We have that $f_1(0) < F(0)$ and since $f_1$ is strictly increasing and continuous on $[0,\gamma_1]$, and $F$ is strictly decreasing and $\lim_{\theta\rightarrow\gamma_1^-} F(\theta) = -\infty$, $F$ and $f_1$ intersect only once in the interval $(0,\gamma_1)$.  A similar argument holds for $f_2$ and $F$ on the interval $(0,\gamma_1)$.  Finally, on the interval $(\gamma_{k-1}, \pi]$, we have $f_2(\gamma_{k-1}) < F(\pi)$ and since $f_2$ decreases and $F$ is strictly increasing on the interval $(\gamma_{k-1},\pi)$ then $f_2$ and $F$ do not intersect there.  On the interval $(\gamma_{k-1},\pi)$, $f_1$ has vertical asymptote at $\theta=\pi$ and is strictly increasing and $F$ is continuous and decreasing on $(\gamma_{k-1},\pi]$.  Thus, in $(\gamma_{k-1},\pi]$, $f_1$ and $F$ intersect only once.  This completes the proof.
\end{proof}

Theorem~\ref{thm:min-eig} now follows from the fact that $\lim_{k\rightarrow\infty} f_1(\theta^+_1) = f_1(0) = \frac{-1+\sqrt{2}}{2}$ and that $\lim_{k\rightarrow\infty} f_2(\theta^-_1) = f_2(0) = \frac{-1-\sqrt{2}}{2}$.  We also obtain the following corollary.

\begin{corollary}
Let $\lambda_{\textup{max}}>0$ and $\lambda_{\textup{min}}<0$ denote the largest and smallest eigenvalues, respectively, of the connected anti-regular graph $A_n$ where $n$ is even.  Then
\[
F(\pi) = \frac{n}{2} < \lambda_{\textup{max}}
\]
and
\[
f_2\left(\tfrac{2(n/2-1)\pi}{n-1} \right) < \lambda_{\textup{min}}.
\]
\end{corollary}
Through numerical experiments, we have determined that the mid-point of the interval $(\gamma_{k-1}, \pi)$, which is $\frac{(4k-3)\pi}{2(2k-1)}$, is a good approximation to $\theta^+_{k}\in (\gamma_{k-1}, \pi)$, that is, 
\[
\lambda_{\textup{max}} \approx F\left( \frac{(4k-3)\pi}{2(2k-1)} \right).
\]
In Table~\ref{tab:max-eig} we show the results of computing the ratio $t_k = \frac{(\theta^+_k - \gamma_{k-1})}{(\pi - \gamma_{k-1})}$ for $k=125, 250, \ldots, 32000$ which shows that possibly $\lim_{k\rightarrow\infty} t_k = \tfrac{1}{2}$.
\begin{table}
\centering
\begin{tabular}{ | c | c | }\hline
$n=2k$ & $t_k$\\
\hline\hline
250&0.5020031290\\
\hline
500&0.5010007838\\
\hline
1000&0.5005001962\\
\hline
2000&0.5002500492\\
\hline
4000&0.5001250123\\
\hline
8000&0.5000625018\\
\hline
16000&0.5000312567\\
\hline
32000&0.5000156204\\
\hline
\hline 
\end{tabular}
\caption{The ratio $t_k = \frac{(\theta^+_k - \gamma_{k-1})}{(\pi - \gamma_{k-1})}$ for $k=125,250,500,\ldots,32000$}
\label{tab:max-eig}
\end{table}

\section{The eigenvalues of large anti-regular graphs}
A graph $G$ is called \textit{bipartite} if there exists a partition $\{X,Y\}$ of the vertex set $V(G)$ such that any edge of $G$ contains one vertex in $X$ and the other in $Y$.  It is known that the eigenvalues of a bipartite graph $G$ are symmetric about the origin.  Figures~\ref{fig:solutions-graph-kmin}-\ref{fig:solutions-graph-kmax} reveal that for the connected anti-regular graph $A_{2k}$ a similar symmetry property about the point $-\tfrac{1}{2}$ is approximately true.  Specifically, if $\lambda\neq \lambda_{\textup{max}}$ is a positive eigenvalue of $A_{2k}$ then $-1 - \lambda$ is approximately an eigenvalue of $A_{2k}$, and moreover the proportion $r\in (0,1)$ of the eigenvalues that satisfy this property to within a given error $\varepsilon > 0$ increases as the number of vertices increases. 

Recall that if  $\lambda^+_1 < \lambda^+_2 < \cdots < \lambda^+_k$ denote the positive eigenvalues of $A_{2k}$ then there exists unique $\theta^+_1 < \theta^+_2 < \cdots < \theta^+_k$ in the interval $(0,\pi)$ such that $\lambda^+_j = f_1(\theta^+_j)$, and if $\lambda^-_{k-1} < \lambda^-_{k-2} < \cdots < \lambda^-_1< -1$ denote the negative eigenvalues of $A_{2k}$ there exists unique $\theta^-_1 < \theta^-_2 < \cdots < \theta^-_{k-1}$ in $(0,\pi)$ such that $\lambda^-_{j} = f_2(\theta^-_j)$ for $j=1,2,\ldots,k-1$.  With this notation we now prove Theorem~\ref{thm:large-k}.

\begin{proof}[Proof of Theorem~\ref{thm:large-k}]
Both $f_1(\theta)$ and $f_2(\theta)$ are continuous on $[0,\pi)$ and therefore are uniformly continuous on the interval $[0,r\pi]$.  Hence, there exists $\delta>0$ such that if $\theta,\gamma \in [0,r\pi]$ and $|\theta-\gamma|<\delta$ then $|f_1(\theta)-f_1(\gamma)|<\varepsilon/2$ and $|f_2(\theta)-f_2(\gamma)|<\varepsilon/2$.  Let $k$ be such that $\frac{2\pi}{2k-1} \leq \delta$ and  let $j^*\in\{1,\ldots,k-1\}$ be the largest integer such that $\frac{2j^*}{2k-1} \leq r$.  Then for all $j\in\{1,\ldots, j^*\}$ it holds that $[\gamma_{j-1}, \gamma_j] \subset [0,r\pi]$.  Let $c_j\in [\gamma_{j-1}, \gamma_j]$  be arbitrarily chosen for each $j\in\{1,\ldots, j^*\}$.  Then $\theta^+_j, \theta^-_j, c_j \in [\gamma_{j-1}, \gamma_j]$ implies that $|f_1(\theta^+_j) - f_1(c_j)| < \varepsilon /2$ and $|f_2(\theta^-_j) - f_2(c_j)| < \varepsilon /2$ for $j\in\{1,\ldots, j^*\}$.  Therefore, if $j\in\{1,\ldots, j^*\}$ then 
\begin{align*}
| \lambda^+_j + \lambda^-_j + 1 | &= |f_1(\theta^+_j) + f_2(\theta^-_j) + 1| \\[2ex]
&= |f_1(\theta^+_j) - f_1(c_j) + f_1(c_j) + f_2(\theta^-_j) - f_2(c_j) + f_2(c_j) + 1| \\[2ex]
&\leq |f_1(\theta^+_j) - f_1(c_j)| + |f_2(\theta^-_j) - f_2(c_j)| + |f_1(c_j)+ f_2(c_j) + 1|\\[2ex]
&= |f_1(\theta^+_j) - f_1(c_j)| + |f_2(\theta^-_j) - f_2(c_j)| \\[2ex]
&< \varepsilon
\end{align*}  
where we used the fact that $f_1(c_j) + f_2(c_j) + 1 = 0$.  This completes the proof for the even case.  As discussed in Section~\ref{sec:odd-case}, the odd case is similar.
\end{proof}

Note that the proportion of $j\in\{1,2,\ldots,k-1\}$ such that $\tfrac{2j}{2k-1}\leq r$ is approximately $r$.  In the next theorem we obtain estimates for  $| \lambda^+_j + \lambda^-_j + 1 |$ using the Mean Value theorem.
\begin{theorem}\label{thm:large-k-estimates}
Let $A_n$ be the connected anti-regular graph where $n=2k$. Then for all $1\leq j\leq k-1$ it holds that
\[
| \lambda^+_j + \lambda^-_j + 1 | \leq \frac{4\pi f_1'(\gamma_j)}{2k-1}.
\]
In particular, for fixed $r\in (0,1)$ and a given arbitrary $\varepsilon > 0$, if $k$ is such that $\frac{4\pi f_1'(r\pi)}{2k-1}< \varepsilon$ then 
\[
| \lambda^+_j + \lambda^-_j + 1 | < \varepsilon 
\]
for all $1\leq j \leq \frac{(2k-1)r}{2}$.
\end{theorem}
\begin{proof}
First note that since $f_1(\theta) + f_2(\theta) = -1$ it follows that $f_2'(\theta) = - f_1'(\theta)$.  The derivative $f'_1$ vanishes at $\theta=0$, is non-negative  and strictly increasing on $[0,\pi)$.  Therefore, by the Mean value theorem, on any closed interval $[a,b] \subset [0, \pi)$, both $f_1(\theta)$ and $f_2(\theta)$ are Lipschitz with constant $K=f_1'(b)$.  Hence, a similar computation as in the proof of Theorem~\ref{thm:large-k} shows that
\[
| \lambda^+_j + \lambda^-_j + 1 | \leq \frac{4\pi  f_1'(\gamma_j)}{2k-1}
\]
for $j=1,2,\ldots, k-1$.  Therefore, if $k$ is such that $\frac{4\pi f_1'(r\pi)}{2k-1}< \varepsilon$ then for $1\leq j \leq \frac{(2k-1)r}{2}$ we have that $\frac{2\pi j}{2k-1}\leq r\pi $ and therefore
\[
| \lambda^+_j + \lambda^-_j + 1 | \leq \frac{4\pi f_1'(\gamma_j)}{2k-1} \leq \frac{4\pi f_1'(r\pi)}{2k-1}< \varepsilon.
\]
\end{proof}

A similar proof gives the following estimates for the eigenvalues with error bounds.
\begin{theorem}\label{thm:eig-estimates-mvt}
Let $A_n$ be the connected anti-regular graph where $n=2k$.  For $1\leq j\leq k-1$ it holds that
\[
|\lambda^+_j - f_1(\gamma_j)| \leq \frac{2\pi f_1'(\gamma_j)}{2k-1}
\]
and
\[
|\lambda^-_j - f_2(\gamma_j)| \leq \frac{2\pi f_1'(\gamma_j)}{2k-1}.
\]
\end{theorem}

We now prove Theorem~\ref{thm:asymptotic-eig}.
\begin{proof}[Proof of Theorem~\ref{thm:asymptotic-eig}]
It is clear that $\{-1,0\}\subset\sigma\subset\bar{\sigma}$.  Let $\varepsilon >0$ be arbitrary and let $y \in [\tfrac{-1+\sqrt{2}}{2}, \infty)$.  Then $y\in\bar{\sigma}$ if there exists $\mu\in\sigma$ such that $|\mu-y|<\varepsilon$.  If $y\in \sigma$ the result is trivial, so assume that $y\not \in \sigma$.  Since $f_1:[0,\pi)\rightarrow  [\tfrac{-1+\sqrt{2}}{2}, \infty)$ is a bijection, there exists a unique $\theta' \in [0, \pi)$ such that $y=f_1(\theta')$.  Let $c\in [0,\pi)$ be such that $\theta' < c < \pi$.  For $k$ sufficiently large, there exists $j\in\{1,\ldots,k-1\}$ such that $\theta' \in [\gamma_{j-1}, \gamma_j]$ and $\tfrac{2j\pi}{2k-1}\leq c$.  Increasing $k$ if necessary, we can ensure that also $\frac{2\pi f_1'(c)}{2k-1}<\varepsilon$.  Then by the Mean value theorem applied to $f_1$ on the interval $I = [\min\{\theta',\theta^+_j\},\max\{\theta',\theta^+_j\}]$, there exists $c_j \in I$ such that
\[
|\lambda^+_j - y| = |f_1(\theta^+_j) - f_1(\theta')| \leq |\theta^+_j-\theta'| f_1'(c_j) <  \frac{2\pi}{2k-1} f_1'(c) < \varepsilon,
\]
where in the penultimate inequality we used the fact that $f_1'$ is increasing and $c_j < c$.  This proves that $y$ is a limit point of $\sigma$ and thus $y\in\bar{\sigma}$.  A similar argument can be performed in the case that $y\in (-\infty, \tfrac{-1-\sqrt{2}}{2}]$ using $f_2$.
\end{proof}


\section{The odd case}\label{sec:odd-case}
In this section, we give an overview of the details for the case that $A_n$ is the unique connected anti-regular graph with $n=2k+1$ vertices.  In the canonical labeling of $A_n$, the partition $\pi = \{\{v_1,v_2\}, \{v_3\}, \{v_4\}, \ldots, \{v_{n}\} \}=\{C_1,C_2,\ldots, C_{2k}\}$ is an \textit{equitable partition} of $A_n$ \cite{CG-GR:01}.  In other words, $\pi$ is the \textit{degree partition} of $A_n$ (we note that this is true for any threshold graph).  The quotient graph $A_n/\pi$ has vertex set $\pi$ and its $2k\times 2k$ adjacency matrix is 
\[
A/\pi = 
\begin{pmatrix}
0&1&0&1&\cdots&0&1\\
2&0&0&1&\cdots&0&1\\
0&0&0&1&\cdots&\vdots&\vdots\\
2&1&1&0&\cdots&\vdots&\vdots\\
\vdots&\vdots&\vdots&\vdots&\ddots&\vdots&\vdots\\
0&\cdots&\cdots&\cdots&\cdots&0&1\\
2&1&1&\cdots&\cdots&1&0
\end{pmatrix}.
\]
In other words, $A/\pi$ is obtained from the adjacency matrix of the anti-regular graph $A_{2k}$ (in the canonical labeling) with the $1$'s in the first column replaced by $2$'s.  It is a standard result that all of the eigenvalues of $A/\pi$ are eigenvalues of $A_n$ \cite{CG-GR:01}.  At this point, we proceed just as in Section~\ref{sec:even-case}.  Under the same permutation \eqref{eqn:perm} of the vertices of $A_n/\pi$, the quotient adjacency matrix $A/\pi$ takes the block form
\[
A/\pi = \begin{pmatrix} 0 & B \\[2ex] C & J-I \end{pmatrix}
\]
where 
\[
C = 
\begin{pmatrix}
&&&&2\\
&&&1&2\\
&&\iddots&\vdots & \vdots\\
&\iddots&&\vdots&\vdots\\
1&1&\cdots&1&2
\end{pmatrix}.
\]
Then
\[
(A/\pi)^{-1} = 
\begin{pmatrix}
-C^{-1}(J-I)B^{-1} & C^{-1} \\[2ex] B^{-1} & 0 
\end{pmatrix}.
\]
After computations similar to the even case, the analogue of \eqref{eqn:Un-alpha} is
\[
\frac{(\alpha^2-1/2)}{\alpha +1} U_{k-1}\left(\frac{h(\alpha)}{2}\right) - \frac{1}{2} U_{k-2}\left(\frac{h(\alpha)}{2}\right) = 0.
\]
After making the substitution $\alpha=\frac{1}{\lambda}$ and simplifying one obtains
\[
\frac{(2-\lambda^2)}{\lambda(\lambda+1)} U_{k-1}(\beta(\lambda)) - U_{k-2}(\beta(\lambda)) = 0
\]
or equivalently
\[
\frac{(2-\lambda^2)}{\lambda(\lambda+1)}  = \frac{U_{k-2}(\beta(\lambda))}{U_{k-1}(\beta(\lambda))}  = \frac{\sin((k-1)\theta)}{\sin(k\theta)}.
\]
The analogue of Theorem~\ref{thm:main-eqn-even} in the odd case is the following.
\begin{theorem}
Let $n=2k+1$ and let $A_n$ denote the connected anti-regular graph with $n$ vertices.  Then $\lambda\neq 0$ is an eigenvalue of $A_n$ if and only if
\begin{equation}\label{eqn:alpha-odd}
\frac{(2-\lambda^2)}{\lambda(\lambda+1)}  = \frac{\sin((k-1)\theta)}{\sin(k\theta)}
\end{equation}
where $\theta  = \arccos\left(\frac{1-2\lambda-2\lambda^2}{2\lambda(\lambda+1)}\right)$.
\end{theorem}

Define the function $g(\lambda) = \frac{(2-\lambda^2)}{\lambda(\lambda+1)}$.  Changing variables from $\lambda$ to $\theta$ as in the even case, and defining $g_1(\theta) = g(f_1(\theta))$, $g_2(\theta) = g(f_2(\theta))$, and in this case $F(\theta) =\frac{\sin((k-1)\theta)}{\sin(k\theta)}$, we obtain the two equations
\begin{align*}
g_1(\theta) &= F(\theta)\\
g_2(\theta) &= F(\theta).
\end{align*}
The explicit expressions for $g_1$ and $g_2$ are
\begin{align*}
g_1(\theta) &= 2+3\cos(\theta) + \sqrt{(\cos\theta+1)(\cos\theta+3)}\\[2ex]
g_2(\theta) &= 2+3\cos(\theta) - \sqrt{(\cos\theta+1)(\cos\theta+3)}.
\end{align*}
The graphs of $g_1, g_2$, and $F$ on the interval $[0,\pi]$ are shown in Figure~\ref{fig:solutions-graph-odd}.  In this case, the singularities of $F$ occur at the equally spaced points 
\[
\gamma_j = \frac{j\pi}{k}, \quad j=1,2,\ldots,k.
\]
\begin{figure}
\centering
\includegraphics[width=120mm,keepaspectratio]{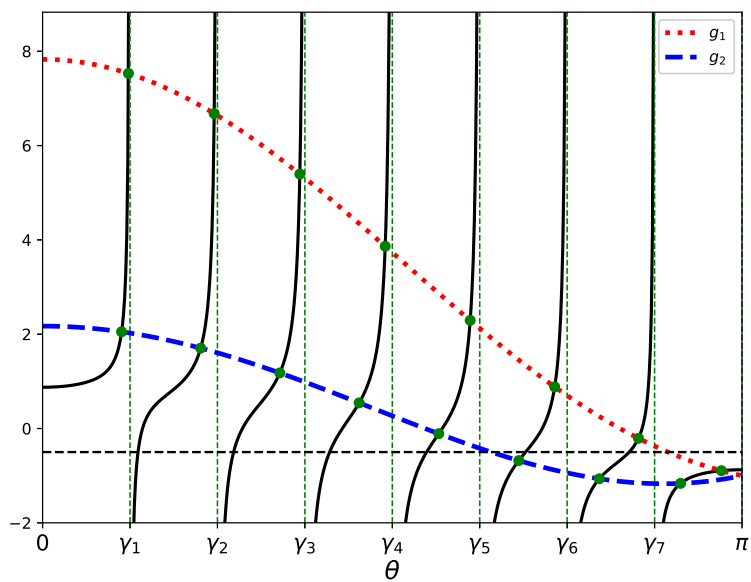}
\caption{Graph of the functions $g_1(\theta), g_2(\theta)$, and $F(\theta)$ (black) for $\theta\in [0,\pi]$ for $k=8$}\label{fig:solutions-graph-odd}
\end{figure}
If $\theta_1^+<\theta^+_2<\cdots<\theta^+_k$ denote the unique points where $F$ and $g_1$ intersect then $f_1(\theta^+_1) < f_1(\theta^+_2) < \cdots < f_1(\theta^+_k)$ are the positive eigenvalues of $A_{2k+1}$.  Similarly, if $\theta_1^-<\theta^-_2<\cdots<\theta^-_k$ denote the unique points where $F$ and $g_2$ intersect then $f_2(\theta^-_k) < f_2(\theta^-_{k-1}) < \cdots < f_2(\theta^-_1)$ are the negative eigenvalues of $A_{2k+1}$.  

Theorems~\ref{thm:interval-bnd}-\ref{thm:large-k} hold for the odd case with now $\lambda = 0$ being the trivial eigenvalue.  Theorem~\ref{thm:intersections},  Theorem~\ref{thm:large-k-estimates},  and Theorem~\ref{thm:eig-estimates-mvt} proved for the even case hold almost verbatim for the odd case; the only change is that the ratio $\frac{2\pi}{2k-1}$ is now $\frac{\pi}{k}$.

\section{The eigenvalues of threshold graphs}\label{sec:general-threshold}
In this section, we discuss how a characterization of the eigenvalues of $A_n$ could be used to characterize the eigenvalues of general threshold graphs.  Let $G$ be a threshold graph with binary creation sequence $b=(0^{s_1},1^{t_1}, \cdots, 0^{s_k},1^{t_k})$, where $0^{s_i}$ is short-hand for $s_i\geq 0$ consecutive zeros, and similarly for $1^{t_i}$.  Let $V(G)=\{v_1,v_2,\ldots,v_n\}$ denote the associated canonical labeling of $G$ consistent with $b$.  The set partition $\pi=\{C_1,C_2,\ldots,C_{2k}\}$ of $V(G)$ where $C_1$ contains the first $s_1$ vertices, $C_2$ contains the next $t_1$ vertices, and so on, is an equitable partition of $G$.  The $2k\times 2k$ quotient graph $G/\pi$ has adjacency matrix
\[
A_\pi = A_{2k} + \diag(0,\beta_1,\ldots, 0, \beta_k)
\]
where $A_{2k}$ is the adjacency matrix of the connected anti-regular graph with $2k$ vertices and $\beta_i = 1 - \frac{1}{t_i}$, see for instance \cite{AB-RM:17}.  The eigenvalues of $G$ other than the trivial eigenvalues $\lambda =-1$ and/or $\lambda=0$ are exactly the eigenvalues of $A_\pi$.  Presumably, the characterization of the eigenvalues of $A_{2k}$ that we have done in this paper will be useful in characterizing the eigenvalues of $A_{\pi}$.  We leave this investigation for a future paper.

\section{Acknowledgements}
The authors acknowledge the support of the National Science Foundation under Grant No. ECCS-1700578.


\end{document}